\documentclass{article}
\usepackage{amssymb,amsmath,amsthm}
\usepackage[T2A]{fontenc}
\usepackage[utf8]{inputenc}
\usepackage[english]{babel}
\usepackage{color}
\usepackage{hyperref} 
\usepackage[normalem]{ulem}
\def\dom{\mathop{\rm dom}}
\def\ds{\displaystyle}

\def\epsilon{\varepsilon}
\def\phi{\varphi}
\newcommand{\cl}{{\rm cl}}

\def\gph{\mathop{\rm gph}}

\renewcommand{\emptyset}{\varnothing}

\def\dom{\mathop{\rm dom}}

\def\gph{\mathop{\rm gph}}
\def\dom{\mathop{\rm dom}}

\def\ball{{I\kern -.35em B}}
\def\epsilon{\varepsilon}
\def\phi{\varphi}

\def\reals{{I\kern-.35em R}} \def\Reals{\overline{I\kern-.35em R}}

\def\qed{\hfill{$\vcenter{\hrule heigh$T_1$pt \hbox{\vrule width1pt height5pt
    \kern5pt \vrule width1pt} \hrule heigh$T_1$pt}$} \medskip}

\def\text#1{\,\;\hbox{#1}\;\,}
\newcommand{\N}{\mathbb{N}}
\newcommand{\R}{\mathbb{R}}

\newcommand{\beq}{\begin{equation}}
\newcommand{\eeq}{\end{equation}}
\newcommand{\be}{\begin{equation}}
\newcommand{\ee}{\end{equation}}

\renewenvironment{proof}
         {\begin{trivlist}\item[
         {\bf Proof.}]}{{\hfill $\square$} \end{trivlist}}

\newtheorem{theo}{Theorem}[section]
\newtheorem{defi}[theo]{Definition}
\newtheorem{rem}[theo]{Remark}
\newtheorem{lemma}[theo]{Lemma}
\newtheorem{prop}[theo]{Proposition}
\newtheorem{examp}[theo]{Example}
\newtheorem{cor}[theo]{Corollary}
\title{Fixed Point Theorems for Set-Valued Maps with Contractive Orbits}
\author{Detelina Kamburova\thanks{Faculty of Mathematics and Informatics, Sofia University,   5, James Bourchier Blvd, 1164 Sofia, Bulgaria, e-mail: detelinak@fmi.uni-sofia.bg;} \thanks{ Institute of Mathematics and Informatics, Bulgarian Academy of Sciences, Acad. G. Bonchev Str., Block 8,
1113 Sofia, Bulgaria;
e-mail: detelinak@math.bas.bg.}}

\date{2026}

\begin{document}

\maketitle

\begin{abstract}
This paper studies set-valued maps that are lower semicontinuous when restricted to orbits in Hausdorff spaces. We introduce two notions of contractive orbits for such maps. The first defines contraction in terms of the topology of the underlying space, while the second is based on a generalized distance function. Fixed point theorems are established for both classes of mappings. We also show that the Hausdorff assumption is essential, as the results generally fail without it. As an application, we generalize Cantor's intersection theorem for sequences of closed nested sets with diameters converging to zero. We derive fixed point theorems for set-valued maps in Hausdorff locally convex vector spaces. The results in premetric spaces we apply to establish fixed point theorems for set-valued maps regular with respect to a generalized distance function.


\end{abstract}

\section{Introduction}

The Banach contraction principle states that in a complete metric space every contraction mapping has a unique fixed point. Moreover, each sequence of successive iterates converges to this fixed point. The existence of an accumulation point of a given sequence is typically ensured by the completeness of the underlying space.  An extensive study of fixed point results that imply completeness in metric and generalized metric spaces can be found in \cite{Cob}.

A definition of a contraction in first countable Hausdorff spaces was given by Kupka in \cite{Kupka}. Kupka continues his work in \cite{Kupka_Sl} for set-valued maps (feebly topologically contractive maps, f-t-contractive) with closed graphs that "shrink" to a unique fixed point, i.e., a strict fixed point. In \cite{MR}, it is proved that if one does not require uniqueness of the fixed point of a single-valued (continuous or closed) mapping, it suffices to consider just two consecutive terms of a given orbit being arbitrarily close to ensure the existence of a fixed point. Other results concerning fixed point theorems for topological contractions can be found in  \cite{MR_B} (in compact $T_1$ spaces).

The LOEV (Long Orbit or Empty Value) principle has been formulated and proved in \cite{Ivanov-Zlateva} in complete metric spaces $(X,d)$. It concerns a set-valued map $S: X\rightrightarrows X$ that "shrinks" to a point $\bar{x}$ outside of the domain of $S$, i.e., $S(\bar{x})=\emptyset$. Applications of the LOEV principle can be found for example in \cite{iv_zla_nme}, \cite{plr}. 
 Its equivalence to $\Sigma$-semicompleteness of generalized metric spaces as well as to other fundamental in variational analysis and optimization results such as Ekeland's variational principle, Caristi's fixed point theorem, Takahashi's minimization principle and Oettli-Thera theorem in generalized metric spaces has been thoroughly studied in \cite{IKZ} and \cite{KZ}.

Another key property of the space used in the proof of the LOEV principle is that the underlying space is Hausdorff. The aim of this paper is to generalize the LOEV statement in Hausdorff spaces. In our main results (Theorem \ref{Hausdorff_Empty} Theorem\ref{strong_point} ) we do not impose any closed-graph assumptions on the set-valued maps we consider, nor compactness of the underlying space. Instead, we assume the maps we work with are lower semicontinuous on a given orbit (see Section 2, Definition \ref{lower_on_orbit}) or satisfy the property $(\bar{\star})$ (a concept inherited from LOEV, see Section 5, Definition \ref{def:bar_star_2}). 

The paper is organized as follows. Section 2 considers topological spaces with no generalized metric defined on them. We define f-t-contractive orbits, establish fixed point theorems for lower semicontinuous maps, and generalize Cantor's intersection theorem for closed nested sets. Adapting the techniques of \cite{MR}, Section 3 establishes sufficient conditions for the existence of strict fixed points in locally convex vector spaces. Section 4 restricts the framework to first countable Hausdorff spaces. We define strong accumulation points for contractive orbits and prove that the Hausdorff assumption is strictly essential and cannot be dropped. Section 5 transitions to premetric spaces, that is, spaces with a generalized distance function that lacks the symmetry property and the triangle inequality. Introduced in the early 20th century, premetric spaces remain an active area of research, see \cite{P-Ch}, \cite{Niem}, \cite{Ned},  \cite{Mas}. We define a $p$-contractive orbit and compare it against the f-t-contractive notion.  If the premetric function satisfies some additional properties, f-t-contractive orbits are $p$-contractive in Hausdorff premetric spaces. The converse is not generally true. Section 6 provides applications to fixed point theorems for set-valued maps that are regular with respect to a lower semicontinuous premetric.





\section{Topological contractions in Hausdorff spaces}

Given a set-valued map $S: X \rightrightarrows X$, a succession of points in $X$ satisfying $x_{i+1}\in S(x_i)$ for $i=1,2,\dots $ is called an $S$-orbit. We say that an $S$-orbit is finite and ends at $x$ if there exists $n \in \N$ such that $x_n=x$ and $S(x)= \emptyset$. If $S(x_i) \neq \emptyset$ for all $i=1,2,\dots $ then the $S$-orbit is infinite. In particular, we consider the stationary orbit $x,x,x, \dots$ as an infinite orbit. Denote by $\dom S := \{x\in X:\ S(x) \neq \emptyset\}$. The graph of $S:X \rightrightarrows Y$, where $X,Y$ are topological spaces, is the set $\gph S$:=$\{(x,y) \in (X \times Y): y \in S(x)\}$.   


\begin{defi}
   \label{lower_on_orbit} Let $(X, \tau)$ be a topological space and $S:X\rightrightarrows X$ be a set-valued map. We say that $S$ is lower semicontinuous on an $S$-orbit $M$ at $x \in \cl M \cap \dom S$  if for every open set $V \subseteq X$ satisfying $S(x) \cap V \neq \emptyset$, there exists an open neighborhood $U \subseteq X$ of $x$ such that:
   $$S(x') \cap V \neq \emptyset \quad \forall x' \in U \cap M.$$
We say that a set-valued map is lower semicontinuous on an $S$-orbit if it is lower semicontinuous with respect to $M$ at each $x \in \cl M \cap \dom S$. 
  \end{defi}

\begin{rem} \label{rem} Note that in a Hausdorff space if

(1) the $S$-orbit is a convergent sequence, it is enough to check the lower semicontinuiuty at the limit point;

(2) an $S$-orbit $M$ is finite then $S$ is trivially lower semicontinuous on $M$.

\end{rem}

\begin{defi} \label{tau-contr} Let $(X, \tau)$ be an arbitrary topological space. Let $S:X \rightrightarrows X$ be a set-valued map. An infinite $S$-orbit $\{x_i\}_{i = 1}^{\infty}$ is said to be feebly topologically contractive (f-t-contractive) if for every open cover $\gamma$ of $X$ there exist $U \in \gamma$ and $i_0 \in \N$ such that $x_{i_0} \in U$ and $S(x_{i_0}) \subseteq U$.
\end{defi}

\begin{lemma} \label{f-t-orbit}
Let $(X, \tau)$ be a Hausdorff space. Let $S: X \rightrightarrows X$ be a set-valued map. Suppose $M:=\{x_i\}_{i=1}^{\infty}$ is an infinite f-t-contractive $S$-orbit. Then there exists a point $\bar{x} \in X$ such that for each neighbourhood $N_{\bar{x}}$ of $\bar{x}$ there exists $i_0 \in \N$ such that $x_{i_0} \in N_{\bar{x}}$ and $S(x_{i_0}) \subseteq N_{\bar{x}}$.
\end{lemma}
\begin{proof}
As in the proof of Theorem 4.5 in \cite{MR}, if  for each $x \in X$ there exists a neighborhood $N_x$ of
$x$ such that for each $i \in \N$ either $S(x_i) \nsubseteq N_x$ or $x_i \notin N_x$. Then the $S$-orbit does not satisfy the definition of
feeble topological contraction for the covering $\{N_x : x \in X\}$.
\end{proof}

\begin{theo} \label{Hausdorff_Empty}
Let $(X, \tau)$ be a Hausdorff topological space. Let $S: X \rightrightarrows X$ be a set-valued map. Suppose $M:=\{x_i\}_{i=1}^{\infty}$ is an infinite f-t-contractive $S$-orbit and $S$ is lower semicontinuous on $M$. Then there exists a point $\bar{x} \in X$ such that $S(\bar{x}) \subseteq \{\bar{x}\}$.
\end{theo}
\begin{proof} 
Let $\bar{x}$ be the point from Lemma \ref{f-t-orbit} for the set-valued map $S$ and the orbit $M$, therefore, for each neighbourhood of $\bar{x}$, $N_{\bar{x}}$ it holds that $N_{\bar{x}} \cap M \neq \emptyset$. We will prove $S(\bar{x}) \subseteq \{\bar{x}\}$. Let $y \in S(\bar{x})$, $y \neq \bar{x}$. Since $X$ is Haudorff there exist open neighbourhoods of  $V\subset X, U_1 \subset X$, of $y$ and $\bar{x}$  respectively such that $U_1 \cap V = \emptyset$. As $S(\bar{x}) \cap V \neq \emptyset$, from lower semicontinuity property of $S$ on $M$ it follows there exists an open neighbourhood of $\bar{x}$, $U_2 \subseteq X$, such that $S(x) \cap V \neq \emptyset$ for all $x \in U_2 \cap M \neq \emptyset$. Denote by $U=U_1 \cap U_2$, an open neighbourhood of $\bar{x}$. Note that $U \cap V =\emptyset$. Let $i_0 \in \N$ be such that $S(x_{i_0}) \subseteq U$ and $x_{i_0} \in U \cap M$. Since $U \subseteq U_2$, it follows  $x_{i_0} \in U_2 \cap M$. So, $S(x_{i_0}) \cap (U \cap V) \neq \emptyset,$ a contradiction.
\end{proof}

\begin{rem}\label{rem2}
If the $S$-orbit is a convergent sequence in a Hausdorff space with a limit point $\bar{x}$ then $S(\bar{x}) \subseteq \{\bar{x}\}$.
\end{rem}

What follows is Cantor’s intersection theorem for a countable family of decreasing, nested, closed subsets in Hausdorff spaces, framed in terms of open coverings. We also note that extensions to admissible spaces (for arbitrary families) and locally convex Hausdorff spaces (for countable families), alongside their equivalence to space completeness and sequential completeness respectively, are established in \cite{Souza_Alves} and \cite{Baisnab_Saha}.

\begin{cor} \label{basic_coro_Cantor_topo}
Let $(X, \tau)$ be a Hausdroff space. Let $\{C_i\}_{i=1}^{\infty}$ be a sequence of closed nested nonempty subsets of $X$. If for every open cover $\gamma$ of $X$ there exist $U \in \gamma$ and $j_0 \in \N$ such that $C_{j_0} \subseteq U$ then there exists $\bar{x} \in X$ such that $\cap_{i=1}^{\infty}C_i=\{\bar{x}\}$.
\end{cor}
\begin{proof}
Define the set-valued map $S$ as
\[ S(x):= \left \{\begin{array}{rl} C_1, & \text{if } x \notin C_1;  \\
 C_{k}, & \text{if there exists} k \in \N: x \in C_{k} \setminus C_{k+1};\\
\displaystyle{\bigcap_{n=1}^{\infty}C_n,} & \text{if } \displaystyle{x \in \bigcap_{n=1}^{\infty}C_n}.\end{array}\right.\] Construct an orbit $M:=\{x_i\}_{i=1}^{\infty}$ as follows
\[\text{choose} x_1 \in C_1, \text{if } x_1 \in \bigcap_{n=1}^{\infty}C_n, \text{then} x_i=x_1 \text{for all} i \in \N, \text{otherwise}\] \[\text{set} k_1:=\sup\{m:x_1 \in C_m\} < \infty,\] 
\[\text{choose} x_2 \in \bigcap_{n=1}^{k_1+1} C_n, \text{if } x_2 \in \bigcap_{n=1}^{\infty}C_n, \text{then} x_i=x_2 \text{for all} i \geq 2, \text{otherwise}\] \[\text{for all} i \in \N \text{set} k_{i+1}:= \sup\{m:x_{i+1} \in C_m\} < \infty, \] 
\[\text{choose} x_{i+2} \in \bigcap_{n=1}^{k_{i+1}+1} C_n, \dots\]

Observe that $x \in S(x)$ for all $x \in M$ and if $x \in C_n$ for some $n \in \N$ then $S(x) \subseteq C_n$. Moreover, $k_i \rightarrow \infty$ as $i\rightarrow \infty$. For every open cover $\gamma$ of $X$ we can always find $j_0 \in \N$ and $U \in \gamma$ such that $C_{j_0} \subseteq U$. Choose $k_{i} \geq j_0$. Therefore $x_{i+1} \in C_{k_i+1} \subseteq C_{j_0}$ and $S(x_{i+1}) \subseteq C_{j_0}$. Therefore, $M$ is a f-t-contractive $S$-orbit. From Lemma \ref{f-t-orbit} there exists a point $\bar{x}$ such that for each neighbourhood $N_{\bar{x}}$ of $\bar{x}$ there exists $i_0 \in \N$ such that $x_{i_0} \in N_{\bar{x}}$ and $S(x_{i_0}) \subseteq N_{\bar{x}}$. Since the sets are nested all but finitely many members of the sequence $M$ are in $S(x_{i_0})$, that is, $M$ is a convergent sequence and $\bar{x}$ is the limit point of $M$.

First, we will prove $\bar{x} \in C_i$ for all $i \in \N$.  Fix $k \in \N$. Once again, all but finitely many members of the sequence $M$ are in $C_k$. Since $C_k$ is a closed set then $\bar{x} \in C_k$. We chose $k$ arbitrarily then $\bar{x} \in C_i$ for all $i \in \N$.

Finally, we demonstrate that $S(\bar{x})$ is a singleton. Lower semicontinuity of $S$ on $M$ is trivial. The assumptions of Theorem \ref{Hausdorff_Empty} are satisfied. Moreover, $S(x) \neq \emptyset$ for all $x \in X$ and therefore \[S(\bar{x})=\{\bar{x}\}=\displaystyle{\bigcap_{n=1}^{\infty}C_n}.\] 
\end{proof}

Assuming the lower semicontinuity of $S$ at each point in $\dom S$ allows us to relax the contraction condition, removing the requirement for it to be restricted to an orbit.
\begin{theo} \label{f-t-contractive_global}
Let $(X, \tau)$ be a Hausdorff space. Let $S: X \rightrightarrows X$ be a set-valued map and suppose 

(1) $S$ is lower semicontinuous at each point in $\dom S$;

(2) if for every open cover $\gamma$ of $X$ there exists $x \in X$ there exists and $U \in \gamma$ such that $S(x) \subseteq U$ and $x \in U$. 

Then there exists a point $\bar{x} \in X$ such that $S(\bar{x}) \subseteq \{\bar{x}\}$.
\end{theo}
\begin{proof} Follows the proof of Lemma \ref{f-t-orbit} and Theorem \ref{Hausdorff_Empty} and is omitted.
\end{proof}

We now proceed with another fixed point theorem derived from assumptions imposed on a given orbit. The proof follows the same steps as the proof of Theorem 2 \cite{Kupka_Sl} concerning the existence of a fixed point. 
\begin{theo} \label{generalization} Let $(X,\tau)$ be an arbitrary topological space. Let $S: X \rightrightarrows X$ be a set-valued map with a closed graph. Let $\dom S=X$ and  $M:=\{x_i\}_i$ be an $S$-orbit. If for every open cover $\gamma$ of $X$ there exists $i_0 \in N$ and $U \in \gamma$ such that $S(x_{i_0}) \subseteq U$ and $x_{i_0+2} \in U$ holds then $S$ has a fixed point.
\end{theo}
\begin{proof} As it was pointed out, we follow the proof of Theorem 2 in \cite{Kupka_Sl}. Consider the space $X \times X$ with the product topology and the set $O = X \times X \setminus \gph S$. Suppose there is no fixed point of $S$ then $O$ is open neighborhood of the diagonal $(x,x) \in X \times X$. 

Consider the following open cover of $X$ $\gamma:=\{V \in X, V - \text{open in} X, V \times V \in O\}$. Then there exists $i_0 \in \N$ and $U \in \gamma$ such that $S(x_{i_0}) \subseteq U$ and $x_{i_0+2} \in U$. 

Since $x_{i_0+1} \in S(x_{i_0})$, then $x_{i_0+2} \in S(x_{i_0+1}) \subseteq S(S(x_{i_0}))$, i.e., $S(S(x_{i_0})) \cap U \neq \emptyset$ and $S(U) \cap U \neq \emptyset$. It follows $\gph S \cap (U \times U) \neq \emptyset$. The last is a contradiction with the definition of $\gamma$.  
\end{proof}
Note that Theorem \ref{generalization} holds in an arbitrary topological space but does not ensure a strict fixed point.

\section{Applications in locally convex vector spaces}

The notations in this section follow that introduced in Section 5 in \cite{MR} with a few additional notations introduced as needed. Let $X$ be a vector space and $\mathbf{S}$ a family of seminorms defined on $X$ that separates points, that is, for each $x \in X \setminus \{0\}$, there exists $s \in \mathbf{S}$ such that $s(x) > 0$.  Let $\tau$ be the topology on $X$ generated by $\mathbf{S}$ which makes $(X, \tau, \mathbf{S})$ a Hausdorff locally convex vector space. Let $\mathbf{E} \subseteq \mathbf{S}$ be a finite set of seminorms on $X$. For each $x \in X$ and each $\varepsilon>0$ denote by
\[V(x;\mathbf{E},\varepsilon):=\{z \in X: s(x-z)<\varepsilon, \text{for all} s \in \mathbf{E}\}.\]
Let $\mathbf{Q}$ be a saturated family of seminorms constructed from $\mathbf{S}$. For each finite set of seminorms $\mathbf{E} \subseteq \mathbf{S}$ define 
\[q_{\mathbf{E}}(x):=\max_{s \in \mathbf{E}} s(x)\] 
and include $q_{\mathbf{E}} \in \mathbf{Q}$. The basic neighbourhoods $V(x;\mathbf{E},\varepsilon)$ can be equivalently written as 
\[V(x;\mathbf{E},\varepsilon):=\{z \in X: q_{\mathbf{E}}(x-z)<\varepsilon\}.\]
\begin{theo} \label{cover} (Theorem 5.1 in \cite{MR})
Let $(X, \tau, \mathbf{S})$ be a Hausdorff locally convex topological vector space. Let $Y  \subseteq X$ be a compact 
set and $\gamma$ be an open cover of $Y$. Then there exist a finite set $\mathbf{E} \subseteq \mathbf{S}$ and
$\varepsilon>0$ such that for each $x \in Y$ there exists $U \in \gamma$ such that $V (x; \mathbf{E},\varepsilon) \subseteq U$.
\end{theo}

\begin{cor} \label{fixed_point_weak}
Let $(X, \tau, \mathbf{S})$ be a Hausdorff locally convex topological vector space. Let $Y \subseteq X$ be a compact set and $S : Y \rightrightarrows Y$ be a set-valued map. If $S$ is lower semicontinuous on an $S$-orbit $M:=\{x_i\}_{i=1}^{\infty}$ such
that 
\[\lim_{i \rightarrow \infty} s(x_i-y) = 0, \text{for all} s \in \mathbf{S} \text{and for all} y \in S(x_i),\]
then there exists a point $\bar{x}$ such that $S(\bar{x}) \subseteq \{\bar{x}\}$.
\end{cor}
\begin{proof} Suppose there is no point $\bar{x} \notin \dom S$. The each $S$-orbit is infinite. 

We follow the proof of Theorem 5.2 in \cite{MR}. Let $\gamma$ be an open cover of $Y$. Let $\varepsilon > 0$ and the finite set
$\mathbf{E} \subseteq \mathbf{S}$ be chosen for $\gamma$ accordingly to Theorem \ref{cover}. For every
$x \in Y$ the set $V (x;\mathbf{E}, \varepsilon)$ is a subset of some set $U$ from the cover $\gamma$. Let $i \in \N$ be
such that  
\[ s(x_i-y) < \varepsilon, \text{for all} s \in \mathbf{E} \text{and for all} y \in S(x_i).\] Then $V(x_i;\mathbf{E},\varepsilon) \subseteq U$ for some $U \in \gamma$.
Hence $S(x_i) \subseteq V(x_i;\mathbf{E},\varepsilon)$ and $\{x_i\} \cup S(x_i) \subseteq U$. Therefore $M$ is f-t-contractive. Since $S$ is also lower semicontinuous on $M$ we can apply Theorem \ref{Hausdorff_Empty} to obtain there exists $\bar{x} \in X$ such that $S(\bar{x}) \subseteq \{\bar{x}\}$.
\end{proof}
\begin{cor} \label{fixed_point_sat}
Let $(X, \tau, \mathbf{S})$ be a Hausdorff locally convex topological vector space and $\mathbf{Q}$ be the corresponding saturated family of seminorms. Let $Y \subseteq X$ be a compact set and $S : Y \rightrightarrows Y$ be a lower semicontinuous set-valued map such that $S(x) \neq \emptyset$ for all $x \in Y$. If for all $S$-orbits $\{x_i\}_{i=1}^{\infty}$ the following series is convergent 
\begin{equation} \label{summable}
\sum_{i=1}^{\infty} q(x_{i+1}-x_i) < \infty, \text{for all} q \in \mathbf{Q}, 
\end{equation}
then there exists a point $\bar{x}$ such that $S(\bar{x}) = \{\bar{x}\}$.
\end{cor}
\begin{proof}

Let $\gamma$ be an open cover of $Y$. Let $\varepsilon > 0$ and the finite set
$\mathbf{E} \subseteq \mathbf{S}$ be chosen for $\gamma$ accordingly to Theorem \ref{cover}. For every
$x \in Y$ the set $V (x;\mathbf{E}, \varepsilon):=\{z \in X: q_{\mathbf{E}}(z-x)<\varepsilon\}$ is a subset of some set $U$ from the cover $\gamma$. Now we will construct an orbit following the steps of Lemma 2.4 in \cite{IKZ}. By our starting assumption $S(x) \neq \emptyset$ for all $x \in X$. Fix $x_1 \in X$. Choose $x_{i+1} \in S(x_i)$ so that \begin{equation}\label{construction}
            q_{\mathbf{E}}(x_{i+1},x_i) > \min\{1, \sup_{y \in S(x_i)}q_{\mathbf{E}}(y-x_i)\}/2.
        \end{equation}
From \ref{summable} it follows $\lim_{i \rightarrow \infty} q_{\mathbf{E}}(x_{i+1},x_i)=0$ and from \ref{construction} we conclude $\sup_{y \in S(x_i)}q_{\mathbf{E}}(y-x_i) \rightarrow 0$ as $i \rightarrow \infty$. 
Moreover, from \ref{summable} and the compactness of $Y$ it follows that each $S$-orbit is a convergent sequence. Then there exists a point $x_L$ and $N \in \N$ such that 
\begin{equation} \label{1}
q_{\mathbf{E}}(x_{i}-x_L) < \varepsilon/2 \text{for all} i \geq N.\end{equation}
Choose $M \in \N$ such that 
\[
\sup_{y \in S(x_{i})}q_{\mathbf{E}}(y-x_{i}) < \varepsilon/2 \text{for all} i \geq M.
\]
Now for $n \geq \max\{M,N\}$
and all $y \in S(x_{n})$ we obtain
\begin{equation} \label{2}
q_{\mathbf{E}}(y-x_L) \leq q_{\mathbf{E}}(y-x_{n})+q_{\mathbf{E}}(x_{n}-x_L)<\varepsilon.
\end{equation}
From \ref{1} and from \ref{2} it follows there exists a point $x_n \in X$ such that
\[\{x_n\} \cup S(x_n) \subseteq V (x;\mathbf{E}, \varepsilon) \subseteq U \in \gamma.\]
Assumptions of Theorem \ref{f-t-contractive_global} are satisfied and since there is no point outside of the domain of $S$ the conclusion of Corollary \ref{fixed_point_sat} follows immediately.
\end{proof}
We proceed with a result that is a starting point for the proof of Ekeland's variational principle, Caristi's fixed point theorem and Takahashi's minimization principle. An alternative approach to proving Ekeland's variational principle in Hausdorff locally convex vector spaces can be found in \cite{Hammel_lc}. Extensions to more general topological settings (e.g., uniform spaces) can be found in \cite{Hammel_un}.

Recall that a function $f: X \to \mathbb{R} \cup \{ + \infty \}$ is lower semicontinuous from above if for every sequence $\{x_i\}_{i=1}^\infty \subset \dom f$ such that $x_i \to x $ and $\{f(x_i)\}_{i=1}^\infty$ is decreasing, i.e. $f(x_i)>f(x_{i+1})$, it holds that $f(x) \leq f(x_i)$ for all $i \in \N$, see \cite{KS}. We denote the cardinality of a set $A$ by $|A|$.
\begin{theo} \label{Ek_loc_conv}
Let $(X,\tau, \mathbf{S})$ be a Hausdorff sequentially complete locally convex topological
 vector space. Let $f: X \rightarrow \R \cup \{+\infty\}$ be an extended-valued proper and sequentially lower semicontinuous from above function, bounded from below. Let $\mathbf{\Lambda}$ be a family of positive numbers, such that $|\mathbf{\Lambda}|=|\mathbf{S}|=|\mathbf{I}|$, where $\mathbf{I}$ is an index set. Consider the set-valued map
\begin{equation} \label{Ekeland_map}
S(x):=\{y \in X: \lambda_k s_k(y-x) \leq f(x)-f(y), \text{for all} k \in \mathbf{I}\}.
\end{equation}
Then there exists $\bar{x}$ such that $S(\bar{x})=\{\bar{x}\}$.
\end{theo}
\begin{proof}
Note that $x \in S(x)$ for all $x \in X$. From the triangle inequality it follows 
\begin{equation} 
\label{subset}
S(S(x)) \subseteq S(x).
\end{equation}
Fix $x_1 \in X$ and consider an arbitrary orbit $\{x_i\}_{i=1}^{\infty}$ starting at $x_1$. Observe that for all $k \in \mathbf{I}$
\begin{equation}
\label{finite_sum}  
\lambda_k \sum_{i=1}^n s_k(x_{i+1}-x_i) \leq f(x_0)-f(x_n) \leq f(x_0)-\inf_X f = \text{const.}
\end{equation}
From \ref{finite_sum} it follows that each orbit is a Cauchy sequence. Indeed, fix an arbitrary neighbourhood $V$ of $\mathbf{0}$. There exists $\varepsilon > 0$ and a finite subset $I \subseteq \mathbf{I}$ such that
\[V:=\{x \in X: s_k(x)<\varepsilon, \text{for all} k \in I\}.\]
Let $\displaystyle \lambda_0:=\min_{k \in I} \lambda_k$. From \ref{finite_sum} it follows for all $k \in I$ there exist $N_k \in \N$ such that for all $m_k > n_k \geq N_k$
$$ \displaystyle
{\lambda_k \sum_{i=n_k}^{m_k-1} s_k(x_{i+1}-x_i) \leq \lambda_0 \varepsilon.}
$$
Let $\displaystyle N_0:=\max_{i=1,\dots,k} N_k$. For all $m>n \geq N_0$ and all $k \in I$ it holds that
$$ 
s_k(x_m-x_n) \leq \sum_{i=n}^{m-1} s_k(x_{i+1}-x_i) \leq \frac{\lambda_0}{\lambda_k} \varepsilon < \varepsilon,
$$
that is, each $S$-orbit $\{x_i\}_{i=1}^{\infty}$ is a Cauchy sequence and from the sequential completeness of the space it converges to a point $x_L$.

Now we will prove that $S$ is lower semicontinuous on each $S$-orbit. According to the Remark \ref{rem} we will have to check it only at the limit point $x_L$. Observe that $\{f(x_{i})\}_{i=1}^\infty$ is decreasing, and since $f$ is supposed to be sequentially lower semicontinuous from above, $f(x_L)\le f(x_{i})$ for all $i\ge 1$. Using \ref{subset}, it suffices to prove $x_L \in S(x_i)$ for all $i \in \N$. For the proof of orbital completeness we follow the steps in Theorem 4, \cite{Hammel_un}. Fix $i$ and choose $n >i$, for all $k \in \mathbf{I}$ we have
\begin{equation} \label{m_n}
s_k(x_n-x_i) \leq \frac{1}{\lambda_k}(f(x_i)-f(x_n)),
\end{equation}
\begin{equation} \label{x_m_n}
s_k(x_L-x_i) \leq s_k(x_L-x_n) + s_k(x_n-x_i).
\end{equation}

Combining \ref{m_n} and \ref{x_m_n} we get
\[s_k(x_L-x_i) \leq s_k(x_L-x_n)+\frac{1}{\lambda_k}(f(x_i)-f(x_n)).\]
Now taking the limit as $n \rightarrow \infty$ we obtain
\[s_k(x_L-x_i) \leq \frac{1}{\lambda_k}(f(x_i)-f(x_L)),\]
and therefore $x_L \in S(x_i)$ for all $i \in \N$.

Now we proceed with a construction an f-t contractive orbit. For all $i \in \N$ set
\[z_{i+1} \in S(z_i): f(z_{i+1}) \leq \inf_{z \in S(z_i)} f(z)+\frac{1}{i}.\]
Since each orbit is a convergent sequence then $z_i \rightarrow z_L \in X$ as $i \rightarrow \infty$. Fix an arbitrary neighbourhood $U$ of $z_L$, that is there exist $\varepsilon>0$ and a finite subset $I \in \mathbf{I}$ such that
\[U:=\{x \in X: s_k(x-z_L)<\varepsilon, \text{for all} k \in I.\}\]  
Let $\displaystyle \lambda:=\min_{k \in I}\lambda_k$ and $i_0$ be such that 
\[
\frac{1}{i_0}<\lambda\frac{\varepsilon}{4}.
\]
We obtain for all $y \in S(z_{i_0})$
\[ \lambda_k s_k(y-z_{i_0}) \leq f(z_{i_0})-f(y) \leq f(z_{i_0})-\inf_{z \in S(z_{i_0})}f(z)<\frac{1}{i_0}<\lambda \frac{\varepsilon}{4}, \text{for all} k \in I,
\]
that is, for all $y \in S(z_{i_0})$
\begin{equation}
\label{set_conv}
s_k(y-z_{i_0})< \frac{\varepsilon}{4}, \text{for all} k  \in I.
\end{equation}
Since $z_i \rightarrow z_L$ as $i \rightarrow \infty$ we can always choose $i_1 \geq i_0$ such that 
\begin{equation}
\label{z_i_z_L}
s_k(z_{i_1}-z_L)< \frac{\varepsilon}{2}, \text{for all} k \in I.
\end{equation}
Therefore $z_{i_1} \in U$. Moreover, for all $y \in S(z_{i_1}) \subseteq S(z_{i_0})$ from \ref{set_conv} and \ref{z_i_z_L} we get
\[s_k(y-z_L)<s_k(y-z_{i_0})+s_k(z_{i_0}-z_i)+s_k(z_i-z_L)<\varepsilon \text{for all} k \in I.\]
For an arbitrary neighbourhood $U$ of $z_L$ we can find $i_1 \in \N$ such that
\[\{z_{i_1}\} \cup S(z_{i_1}) \in U,\]
therefore $\{z_i\}_{i=1}^{\infty}$ is an f-t-contractive $S$ orbit on which $S$ is lower semicontinuous. Put $\bar{x}=z_L$ and apply Theorem \ref{Hausdorff_Empty} to obtain $S(\bar{x})=\bar{x}$. 
\end{proof}


\section{Topological contractions in first countable Hausdorff spaces}

Throughout this section, we work within the framework of first-countable spaces, where topological properties can be fully characterized by sequences. 

\begin{defi} \label{strong_acc_point}
Let $(X, \tau)$ be a first countable topological space. Let $S:X \rightrightarrows X$ be a set-valued map. An infinite $S$-orbit $\{x_i\}_{i = 1}^{\infty}$ is said to have a strong accumulation point $\bar{x} \in X$ if any neighborhood of $U_{\bar{x}}$ of $\bar{x}$ is such that $\{x_i\} \cup S(x_i) \subseteq U_{\bar{x}}$ for infinitely many $i$'s, i.e., there is a subsequence $\{x_{i_k}\}_{k= 1}^{\infty}$ such that
 \[x_{i_k} \rightarrow \bar{x}\] 
 and whenever $a_k \in S(x_{i_k})$
 \[a_k \rightarrow \bar{x} \text{as} k \rightarrow \infty,\]
 as well. 
\end{defi}
Lemma \ref{f-t-orbit} and Theorem \ref{Hausdorff_Empty} can be restated as follows.
\begin{lemma} \label{conv_o_orbital}
 Let $(X, \tau)$ be a first countable Hausdorff space. Let $S: X \rightrightarrows X$ be a set-valued map and the $S$-orbit $M:=\{x_i\}_{i=1}^{\infty}$ is an infinte f-t-contractive $S$-orbit. Then there exists a strong accumulation point of $M$.
\end{lemma} 

\begin{proof}
We follow the proof of Lemma 1 in \cite{Kupka}. There are two cases.

\smallskip

Case 1. There exists $\bar{x}$ such that for all $U_{\bar{x}}$ the set 
\[I_{U_{\bar{x}}}:=\{i \in \N: \{x_i\} \cup S(x_i) \subseteq U_{\bar{x}}\}\]
is infinite. Then, $\bar{x}$ is a strong accumulation point.

\smallskip

Case 2. For all $x \in X$ there exists $U_{x}$ such that
\[I_{U_{x}}:=\{i \in \N: \{x_i\} \cup S(x_i) \subseteq U_{x}\}\]
is finite or empty, i.e., for all $x \in X$ either

\smallskip

2.1. $I_{U_{x}}=\emptyset$, i.e., for all $i \in \N$ $x_i \notin U_{x}$ or $S(x_i) \setminus U_x \neq \emptyset$;

2.2. $I_{U_{x}} \neq \emptyset$. 

If there exists $x_n \in M$ such that $S(x_n)=\{x_n\}$ then $S(x_i)=\{x_i\}=\{x_n\} \subseteq U_{x_n}$ for all $i \geq n$, and the result follows.

Otherwise, for all $i \in \N$ there exists $a_i \in S(x_i)$, $a_i \neq x_i$. We will construct an open cover with no element containing both $x_i$ and $S(x_i)$ for some $i \in \N$. For every pair $(x_i, a_i)$, $i \in I_{U_x}$ we will remove either $x_i$ or $a_i$ from $U_x$. For all $x \in X$ define the set

\[V_x:=\left \{\begin{array}{rl} \emptyset, & \text{if } I_{U_x}=\emptyset;  \\
\{a_j\}\cup\{x_i\}_{i \in I_{U_x}, i \neq j}, & \text{if } x = x_j \text{for some} j \in I_{U_x}; \\
\{x_i\}_{i \in I_{U_x}}, & \text{if } x \neq x_i \text{for each} i \in I_{U_x}. \\
\end{array}\right.\]

Since $I_{U_x}$ is finite or empty for all $x \in X$ then $V_x$ is finite or empty for all $x \in X$. Therefore $V_x$ is a closed set and $U_x \setminus V_x$ is an open neighborhood of $x$. Consider the open cover of $X$ $\gamma:=\{Z_x:=U_x \setminus V_x: x \in X\}$. There is no element $Z_x$ of $\gamma$ such that $x_i \in Z_x$ and $S(x_i) \subseteq Z_x$. Therefore every $\tau$-contractive $S$-orbit has a convergent subsequence $\{x_{i_k}\}_k$, $x_{i_k} \rightarrow \bar{x}$, as $k \rightarrow \infty$. Moreover, for every neighborhood $U_{\bar{x}}$ of $\bar{x}$ and for all but finitely many $k \in \N$ it holds that 
$$ 
\{x_{i_k}\} \cup S(x_{i_k}) \subseteq U_{\bar{x}}.
$$
\end{proof}

\begin{theo} \label{Hausdorff_Empty_1} 
Let $(X, \tau)$ be a first countable Hausdorff space. Let $S: X \rightrightarrows X$ be a set-valued map. Suppose $M:=\{x_i\}_{i=1}^{\infty}$ is an infinite f-t-contractive $S$-orbit and $S$ is lower semicontinuous on $M$. Then each strong accumulation point $\bar{x}$ of $M$ is such that $S(\bar{x}) \subseteq \{\bar{x}\}$. 
\end{theo}
\begin{proof} Follows the proof of Theorem \ref{Hausdorff_Empty} and is omitted.
\end{proof}
We now demonstrate that, unlike in Theorem \ref{generalization}, the Hausdorff assumption in Theorem \ref{Hausdorff_Empty_1} is essential and cannot be dropped. 
\begin{lemma} \label{seq_constr} Let $(X, \tau)$ be a first countable $T_1$ space. If $X$ is not Hausdorff, then there exists a sequence $M:=\{x_i\}_{i=1}^{\infty} \subseteq X$ such that $x_i \neq x_{j}$ for all $i \neq j$ and  $\bar{x} \neq \bar{y}$, $x_i \neq \bar{x}$ and $x_i \neq \bar{y}$ for all $i,j \in \N$ with $\{x_i\}_{i=1}^{\infty}$ convergent to both $\bar{x}$ and $\bar{y}$.   
\end{lemma}
\begin{proof}
Since $X$ is not Hausdorff there exist two points $\bar{x} \neq\bar{y}$ such that $U_{\bar{x}} \cap U_{\bar{y}} \neq \emptyset$ for any pair of neighborhoods $U_{\bar{x}}$, $U_{\bar{y}}$ of $\bar{x}$ and $\bar{y}$, respectively. 

Since $X$ is a $T_1$ space, there exist neighborhoods $\tilde{U}_{\bar{x}}$ of $\bar{x}$ and $\tilde{U}_{\bar{y}}$ of $\bar{y}$ such that $\bar{y} \notin \tilde{U}(\bar{x})$, and $\bar{x} \notin \tilde{U}(\bar{y})$. Denote by $\{L_n\}_{n = 1}^{\infty}$, and $\{W_n\}_{n = 1}^{\infty}$ a nested countable local base at the points $\bar{x}$ and $\bar{y}$, respectively, with $L_1 \subseteq \tilde{U}_{\bar{x}}$,  $W_1 \subseteq \tilde{U}_{\bar{y}}$,  and $L_{i+1} \subseteq L_{i}$, $W_{i+1} \subseteq W_{i}$ for every $i=1,2,3,\dots$. 

Construct a sequence $M:=\{x_i\}_{i = 1}^{\infty}$ in the following way $x_1 \in L_1 \cap V_1$. For $i = 2, 3, \dots$, $x_i \in L_{k_i} \cap V_{k_i}$, where $k_i:=\max\{n \in \N: x_{i-1} \notin L_n \cap V_n\}$. Note that $k_i$ is a finite number, otherwise we get a contradiction with the separation axiom $T_1$. The sequence constructed in this way is such that $x_i \neq x_{j}$, for all $i,j \in \N$, $i \neq j$, $\bar{x} \neq \bar{y}$ and $x_i$ converges to both $\bar{x}$ and $\bar{y}$ as $i \rightarrow \infty$. 
\end{proof}

\begin{prop} 
Let $(X, \tau)$ be a first countable $T_1$ space. If $X$ is not a Hausdorff space, then there exists a set-valued map $S$ and an $S$-orbit $M:=\{x_i\}_{i=1}^{\infty}$ with a strong accumulation point $\bar{x}$, such that $S$ is lower semicontinuous on $M$ at $\bar{x}$ and $S(\bar{x}) \nsubseteq \{\bar{x}\}$.
\end{prop}
\begin{proof} Let $M:=\{x_i\}_{i = 1}^{\infty}$ be the sequence constructed in Lemma \ref{seq_constr}. Let $\bar{x}$ and $\bar{y}$ be the limit points of $M_0$ as defined in Lemma \ref{seq_constr}. Consider the following set-valued map:

\begin{equation}
S(x):= \left \{\begin{array}{rl} M, & \text{if } x \notin M \cup \{\bar{x}\} \cup \{\bar{y}\};  \\
\{x_{i+1}\}, & \text{if } x = x_{i} \text{for some} x_{i}  \in M; \\
\{\bar{y}\}, & \text{if } x=\bar{x};\\
\{\bar{x}\}, & \text{if } x=\bar{y}.\end{array}\right.
\end{equation}

The infinite $S$-orbit $\{x_i\}_{i = 1}^{\infty}$ is f-t-contractive and moreover, $S$ is lower semicontinuous on $M$ at $\bar{x}$ and $\bar{y}$. For its strong accumulation points $\bar{x}$ and $\bar{y}$ it holds that $S(\bar{x})=\{\bar{y}\} \nsubseteq \{\bar{x}\}$ and $S(\bar{y})=\{\bar{x}\}\nsubseteq \{\bar{y}\}$.
\end{proof}

\begin{cor}  Let $(X, \tau)$ be a first countable $T_1$ space. Assume that for any set-valued map $S: X \rightrightarrows X$ and any f-t-contractive $S$-orbit $M=\{x_i\}_{i=1}^{\infty}$, the condition $S(\bar{x}) \subseteq \{\bar{x}\}$ holds for every strong accumulation point $\bar{x}$ at which $S$ is lower semicontinuous on $M$. Then, $X$ is a Hausdorff space. \end{cor}

\section{Contractions in premetric spaces}

Consider a first countable topological space $(X,\tau)$ and a function $p:X\times X\to \R^+$ with the property \[\text{(P1)} \;  p(x,y)=0 \Leftrightarrow  x=y.\] We will call the triple $(X, \tau,p)$ a premetric space. Note that in \cite{Mas} the author defines the premetric as a two-variable non-negative function $p$ that satisfies the weaker assumption $p(x,x)=0$ for any $x \in X$. Consider the following additional properties
    
    (P2) if $p(x,x_n)\to 0$, then $x_n\to x$;
    
    (P3) if $x_n\to x$, then $p(x_{n+1},x_{n})\to 0$.

    (P4) $p(x, \cdot)$ is continuous for every fixed $x \in X$.
    
In $\cite{IKZ}$ the authors define the premetric as a two-variable non-negative function $p$ that satisfies (P1) and (P4). Spaces $(X, \tau, p)$ with function $p$ satisfying properties (P2) and (P3) and the additional assumption $(X, \tau)$ is Hausdorff are also considered in \cite{IKZ}. We always assume property (P1). For the sake of clarity and simplicity we will denote by 
\[p:X\times X\to \R^+ \text{a function that satisfies at least property (P1),}\]
\[h:X\times X\to \R^+ \text{a function that satisfies at least properties (P1), (P2) and (P3),}\]
and by 
\[g:X\times X\to \R^+ \text{a function that satisfies at least properties (P1) and (P4).}\]
For a sequence $\{x_i\}_{i=1}^{\infty} $ in a premetric space $(X,\tau,p)$, the sum $\ds  \sum_{i=1}^{\infty} p(x_{i+1},x_i)$ can be considered as the $p$-length of the sequence. If a sequence $\{x_i\}_{i=1}^{\infty}$ is infinite and has a finite $p$-length, it will be called a $\Sigma_p$-Cauchy sequence. A premetric space is called  $\Sigma_p$-semicomplete if every $\Sigma_p$-Cauchy sequence in $X$ has a convergent subsequence and $\Sigma_p$-complete if every $\Sigma_p$-Cauchy sequence in $X$ is convergent, see \cite{IKZ} and \cite{Suzuki-2018}.

\begin{prop}
    \label{prop:h-def-top}
    Let $(X,\tau,h)$ be a premetric space. Then
    
    \emph{(P2')} $x_i\to x \iff h(x,x_i)\to 0$;
    
   \emph{(P3')} if $x_i\to x$, then $h(x_{i+1},x_i)\to 0$ and  $h(x_{i},x_{i+1})\to 0$; 
\end{prop}
\begin{proof}
  The proof of (P2') follows the lines of the proof of Proposition 2.6 in \cite{IKZ} and will be skipped here. 

To prove (P3') we need to prove only that $x_i\to x$ implies $h(x_{i},x_{i+1})\to 0$. Consider the sequence
 $$
        x_2,x_1,x_4,x_3,\ldots,x_{n+1},x_n,\ldots.
    $$
  It also converges to $x$ and by (P3) it follows $h(x_{i},x_{i+1})\to 0$.
\end{proof}

\begin{prop}
\label{$T_1$_space}
 If $(X,\tau,h)$ be a premetric space, then $(X,\tau)$  is a $T_1$ space.
\end{prop}
\begin{proof}
For any $x \in X$, denote by $\{U_i(x)\}_{i=1}^{\infty}$ a countable nested local base of $x$. Suppose that $(X,\tau)$ is not a $T_1$ space, i.e., there are points $z,y \in X$, $z \neq y$ such that $y \in U_i(z)$ for each $i\in\mathbb{N}$. Consider the alternating sequence
\[
y,z,y,z\dots,y,z,\dots
\]
which converges to $z$. From (P3') we have $h(z,y)=0$. From (P1) it follows that $z=y$, a contradiction.
\end{proof}

Here we give an example of a premetric space $(X, \tau, h)$ which is not Hausdorff.
\begin{examp} \label{premetric_T1} Consider the line with two origins $(X, \tau)$, see \cite{Munkres}. The space X is the union of the set $\R \setminus \{0\}$ and the two-point set $\{a\}\cup\{b\}$, $a \neq b$. The space $(X, \tau)$ is first countable and the local base at each point is given by all open intervals in $\R$ that do not contain 0, along with all the sets of the form $(-x,0) \cup \{a\} \cup (0,x)$ and $(-x,0) \cup \{b\} \cup (0,x)$, for $x \in \R$, $x > 0$. This space is $T_1$, but not Hausdorff. We can define a function $h$ as follows $h(x,y)=h(y,x)=|x-y|$, if $x,y \in \R \setminus \{0\}$, $h(x,y)=p(y,x)=|x|$, if $x \in \R \setminus \{0\}$ and $y \in \{a\} \cup \{b\}$ and $h(a,b)=h(b,a)=1$. 
\end{examp}
For a set $C\subseteq X$, $C \neq \emptyset$, we will denote by $p_{C}(x):=\sup_{y\in C} p(y,x)$.

\begin{defi} 
Let $(X, \tau, p)$ be a premetric space. Let $S:X \rightrightarrows X$ be a set-valued map. An infinite $S$-orbit $\{x_i\}_{i=1}^{\infty}$ is said to be p-contractive if there is a convergent subsequence $\{x_{i_k}\}_{k=1}^{\infty}$ and $p_{S(x_{i_k})}(x_{i_k}) \rightarrow 0$ as $k \rightarrow \infty$.
\end{defi}

 \begin{examp} \label{Examp_Ek}   Let $(X,\tau,p)$ be a $\Sigma_p$-semicomplete premetric space. Let $f: X \rightarrow \R \cup \{+\infty\}$ be a proper, bounded below function. Consider the map defined as
$$S(x):=\{y \in X: p(y,x) < f(x)-f(y)\};$$
if each orbit is infinite then, there exists a p-contractive $S$-orbit.
Note that each infinite $S$-orbit has a convergent subsequence. Indeed,
\begin{equation}
\label{finite_sum_pr}  
\displaystyle{\Sigma_{i=1}^n p(x_{i+1},x_i) \leq f(x_1)-f(x_n) \leq f(x_1)-\inf_X f = \text{const.}}
\end{equation}
By $\Sigma_p$-semicompleteness of the space it follows that each $\{x_i\}_{i=1}^{\infty}$ has a convergent subsequence.
Now we will construct an orbit that is p-contractive following the steps of Lemma 2.4 in \cite{IKZ}. By our starting assumption $S(x) \neq \emptyset$ for all $x \in X$. Fix $x_1 \in X$. Choose $x_{i+1} \in S(x_i)$ so that \begin{equation}\label{eq:d}
            p(x_{i+1},x_i) > \min\{1, p_{S(x_i)}(x_i)\}/2.
        \end{equation}

From \ref{finite_sum_pr} it follows $\lim_{i \rightarrow \infty} p(x_{i+1},x_i)=0$ and from \ref{eq:d} we conclude $p_{S(x_i)}(x_i) \rightarrow 0$ as $i \rightarrow \infty$.
\end{examp}
Because the contraction is not necessarily compatible with the underlying topology in general, we require a property stronger than lower semicontinuity. In what follows, we study the behavior of orbits of set-valued maps under the assumption of property $(\bar{\star})$, defined below. 
\begin{defi}
    \label{def:bar_star_2}
 Let $(X, \tau)$ be a first countable space and $S:X\rightrightarrows X$ be a set-valued map. We say that an infinite $S$-orbit $\{x_i\}_{i = 1}^{\infty}$ satisfies the property $(\bar{\star})$ if for each subsequence $\{x_{i_k}\}_{k=1}^\infty$ converging to $x$, and each $y\in S(x)$, $y \neq x$, it holds that
    $$
        y \in S(x_{i_{k}}),\quad\forall k\in\mathbb{N}.
    $$
    \end{defi}
 Definition \ref{def:bar_star_2} is inspired by Definition 2.9 in \cite{IKZ}. Let us note that in Definition 2.9 in \cite{IKZ} the condition is imposed on any orbit that starts at a given point and for $S$ satisfying $x \notin S(x)$ for all $x$.

\begin{prop} \label{t=>p}  Let $(X, \tau, h)$ be a Hausdorff premetric space. Let $S:X \rightrightarrows X$ be a set-valued map. Let $M:=\{x_i\}_{i=1}^{\infty}$ be an infinite f-t-contractive $S$-orbit. Then $M$ is a h-contractive $S$-orbit.
\end{prop}
\begin{proof} Since $\bar{x}$ is a strong accumulation point of $M$ then there exists a subsequence $\{x_{i_k}\}_{k=1}^{\infty}$ and whenever $a_k \in S(x_{i_k})$  $a_k \rightarrow \bar{x}$, i.e., the sequence
\[x_{i_1},a_1,x_{i_2},a_2,\dots,x_{i_k},a_k,\dots\]
converges to $\bar{x}$.
What remains to be proved is $h_{{S(x_{i_k}})}(x_{i_k}) \rightarrow 0$ as $k \rightarrow \infty$.

Fix $\varepsilon>0$. For all $k$ there exists $a_k^{\varepsilon} \in S(x_{i_k})$ such that
\begin{equation} \label{sup_eps}
h_{S(x_{i_k})}(x_{i_k})<h(a_k^{\varepsilon},x_{i_k})+\varepsilon.
\end{equation}

From (P3') it follows that $h(a_k^{\varepsilon}, x_{i_k}) \rightarrow 0$ as $k \rightarrow \infty$. Pick $K \in \N$ such that for all $k \geq K$ it holds that $h(a_k^{\varepsilon}, x_{i_k}) < \varepsilon$. From \ref{sup_eps} it follows that 
\[
h_{S(x_{i_k})}(x_{i_k})<2\varepsilon. 
\]
Since we chose $\varepsilon$ arbitrarily then $h_{S(x_{i_k})}(x_{i_k}) \rightarrow 0$ as $k \rightarrow \infty$. Therefore, $\bar{x}$ is a p-contractive accumulation point of $M$.
\end{proof}
The converse, in general, is not true, see Example \ref{Moore_plane}.

\begin{theo} \label{strong_point}Let $(X, \tau, h)$ be a Hausdorff premetric space. Let $S: X \rightrightarrows X$ be a set-valued map. Suppose $M:=\{x_i\}_{i=1}^{\infty}$ is a h-contractive $S$-orbit that satisfies property  $(\bar{\star})$. Then there exists an accumulation point $\bar{x}$ of $M$ such that $S(\bar{x}) \subseteq \{\bar{x}\}$.
\end{theo}
\begin{proof} 
Assume the contrary, i.e., that there exists a set-valued map $S:X\rightrightarrows X$ with an infinite $S$-orbit $M:=\{x_i\}_{i=1}^{\infty}$ satisfying the property $(\bar{\star})$ with convergent subsequence $\{x_{i_k}\}_k \rightarrow \bar{x}$ satisfying $h_{S(x_i)}(x_i) \rightarrow 0$, and some $y\in S(\bar{x})$, $y\neq \bar{x}$. By property $(\bar{\star})$ of $M$, there exists a subsequence $\{x_{{i_k}_j}\}_{j=1}^{\infty}$ such that $y\in S(x_{{i_k}_j})$ for all $j\in \N$. Because of $y\in S(x_{{i_k}_j})$ and $h_{S(x_{i_k})}(x_{i_k}) \rightarrow 0$ we have that $h(y,x_{{i_k}_j}) \rightarrow 0$. By (ii)  $\{x_{{i_k}_j}\}_j\to y$. But  $\{x_{{i_k}_j}\}_j\to \bar{x}$ and $\bar{x} \neq y$. Since $X$ is a Hausdorff space, this yields a contradiction.
\end{proof}

 Here are two examples that illustrate Theorem \ref{strong_point}. The first example shows that the Hausdorff property is again crucial for Theorem \ref{strong_point}.

\begin{examp} Let $(X ,\tau, h)$ be the premetric defined in example \ref{premetric_T1}. Define a sequence $\{x_i\}^{\infty}_{i=1}$ as follows $x_i=1/i$, for all $i \geq 1$, $i \in \N$. Set $M:=\{x_i\}^{\infty}_{i=1}$ Consider the set-valued map $S:X \rightrightarrows X$ defined as
\[
S(x):= \left \{\begin{array}{rl} M, & \text{if } x \notin M \cup \{a\} \cup \{b\};  \\
\{x_{i+1}\} \cup \{a\}\cup  \{b\}, & \text{if } x = x_{i} \text{for some} x_{i}  \in M; \\
\{b\}, & \text{if } x=a;\\
\{a\}, & \text{if } x=b.\end{array}\right.
\]
Consider the $S$-orbit which consists of the points in $M$. Property $(\bar{\star})$ holds for this orbit. Note that $x_i \rightarrow a$ and $x_i \rightarrow b$ as $i \rightarrow \infty$. From (P2') it follows that
\[h(a,x_i) \rightarrow 0, \; h(b,x_i) \rightarrow 0,\] and from (P3) it follows that \[h(x_{i+1},x_i) \rightarrow 0.\]
Therefore $h_{S(x_{i})}(x_{i}) \rightarrow 0$. But $\{b\}=S(a) \nsubseteq \{a\}$ and $\{a\}=S(b) \nsubseteq \{a\}$.
\end{examp}
 

\begin{examp} \label{Moore_plane}
The Niemytzky's tangent disc topology (also known as Moore plane, see \cite{Counterexamples}, Example 82) is defined on the closed upper-half plane $X=\{(x,y) \in \R^2: y \geq 0\}$. A countable local  base $L_n(x,y)$  at each point $(x,y) \in X$ is given by 
\[L_n(x,y):=\left\{ \begin{array}{ll}
\left\{(t,z) \in R^2: (t-x)^2+(z-y)^2<\frac{1}{n^2}\right\}, & \text{if} \; y>0;\\[10pt]
\{(x,y)\} \cup\left\{(t,z) \in \{R^2: (t-x)^2+(z-\frac{1}{n})^2<\frac{1}{n^2}\right\}, & \text{if} \; y=0.\\
\end{array} \right.\]

 Moore plane $(X,\tau)$ is a first countable Hausdorff space which is not metrizable, but it is a Hausdorff premetric space with the additional properties (P2) and (P3) (see \cite{IKZ}). 
 
 
\smallskip

Consider the map
\[S(x,y):=\left\{\begin{array}{ll}
\{(\frac{x}{2},\frac{x}{2})\}, & \text{if} \; y=0;\\[10pt]
\{(x,0)\}, & \text{if} \; y \neq 0.\\
\end{array} \right.\]
The map $S$ satisfies the assumptions of Theorem \ref{strong_point}. Each infinite $S$-orbit  has a convergent to $(0,0)$ subsequence and $S(0,0)= \{(0,0)\}$.
\end{examp}

\section{Application to fixed point theorems for set-valued maps regular with respect to a premetric}

The set valued map $S$ from Example \ref{Examp_Ek} is a starting point for the proof of Ekeland's variational principle. If the function $f$ is lower semicontinuous from above and the premetric satisfies (P4) it can be shown that each $S$-orbit satisfies property $(\bar{\star})$.  A detailed proof of Ekeland's variational principle in premetric spaces can be found in \cite{KZ}. Note that if we can define a premetric that satisfies property (P4) then the underlying space must be Hausdorff, see \cite{KZ}. 

\begin{theo} \label{th:Ekeland} (Theorem 3.1, \cite{KZ}) 
Let $(X,\tau,g)$ be a  $\Sigma_g$-semicomplete space. Let the function $f: X \rightarrow \R \cup \{+\infty\}$ be proper,  bounded below and lower semicontinuous. Let $\varepsilon>0$ and $\lambda>0$ be arbitrary.   Then
\begin{itemize}
\item[\emph{(j)}] there exists $v \in X$ such that for all $x \in X$
\begin{eqnarray}
f(v) \leq f(x)+\frac{\varepsilon}{\lambda} g(x,v).  \label{eq:Ekeland}
\end{eqnarray}
\end{itemize}
For any $x_0 \in X$  such that $f(x_0) \le \inf_X f + \varepsilon $ and  for the multi-valued map $S :X \rightrightarrows X$ defined as
\[S (x):=\left\{y \in X: f(y)<f(x)-\frac{\varepsilon}{\lambda} g(y,x)\right\}\]
there exists an $S$-orbit starting at $x_0$ and ending at some $v$ with a $g$-length $\ds L(v,x_0):=\sum^{\infty}_{n= 0}g(x_{n+1},x_n)$ satisfying
\begin{itemize}
\item[\emph{(jj)}] $L(v, x_0) \leq \lambda$;
\item[\emph{(jjj)}] $\ds f(v) \leq f(x_0) - \frac{\varepsilon}{\lambda} L(v,x_0)$.
\end{itemize}
\end{theo}




We now extend the fixed-point result of Theorem 7.44 for regular maps in metric spaces in \cite{Io} to the setting of premetric spaces. Extensions of this result in quasi-metric and $(q_1,q_2)$-quasi-metric spaces can be found in \cite{TT1}, \cite{TT2} and \cite{asym}. Further results concerning fixed points in $(q_1, q_2)$-quasi-metric spaces can be found in \cite{IDRZ}. By $B(x,r)$ we denote the "closed ball" of radius $\rho$ and center $x \in X$ in a premetric space $(X, \tau, g)$ defined as follows
$$B(x,\rho):=\{y \in X: g(y,x) \leq \rho\};$$

In the absence of the triangle inequality for the generalized distance $g$, we cannot utilize standard local arguments, consequently, we define the regularity condition globally.
\begin{defi} \label{reg_rate} Let $(X, \tau, g)$ and $(Y, \tau', g')$ be premetric spaces. A  multi-valued map  $S:X \rightrightarrows Y$ is called \emph{regular} with respect to $g$ and $g'$ if there is $r>0$ such that for all $(x,y) \in \gph S$ and all $t>0$,
 \[B(y,rt) \subset S(B(x,t)).\]
The upper bound of all such $r>0$ is called the \emph{regularity rate} of $S$.
\end{defi}
\begin{rem} 
 We make two remarks regarding Definition \ref{reg_rate}.
 
(1) While there are three equivalent notions of regularity in the metric case (see \cite{Io}), this paper does not aim to redefine them for premetric spaces. Instead, we utilize just one of them to establish a fixed-point theorem.

(2) In metric and quasi-metric spaces, Theorem 6.3 is associated with a weaker notion of regularity known as orbital regularity. We do not address it here since it is again outside of the scope of this work.
\end{rem}

\begin{theo} \label{th:Fixed_point_regularity}
Let $(X, \tau, g)$ be a premetric space. Suppose, additionally that $g$ is a lower semicontinuous function in $X \times X$. Let $S:X \rightrightarrows X$ be a set-valued map with closed graph and regular with respect to $g$ with regularity rate $r>1$. Then $S$ has a fixed point.
\end{theo}

\begin{proof}
 Follows part of the proof of Theorem 7.44 in \cite{Io}. Choose an arbitrary point $(\bar{x},\bar{y}) \in \gph  S$. If $\bar{x}=\bar{y}$, we are done. Otherwise, choose $\xi>0$ and $1<r_1<r$ such that $r_1\xi<1$. Consider the following function 
 \[g_\xi((x,y),(u,v))=\max\{g(x,u),\xi g(y,v)\}.\]
 Consider the space $(\gph S, \tau', g_\xi)$ where $\tau'$ is the topology on $\gph S$ induced by the product topology of $X \times X$. We will prove $(\gph S, \tau', g_\xi)$ is a $\Sigma_{g_{\xi}}$-semicomplete premetric space with premetric function satisfying (P4).

 (P1) $g_\xi((x,y),(u,v))=0 \Leftrightarrow  (x,y)=(u,v)$ is trivial;

 (P4) $g_{\xi}((x,y), \cdot)$ for every fixed $(x,y) \in \gph F$ as a maximum of two continuous functions is continuous.

    
    
 What remains to be proved is $\Sigma_{g_{\xi}}$-semicompleteness. Consider a sequence $\{(x_i,y_i)\}^{\infty}_{i=1} \subseteq \gph S$ such that 
\[\sum_{i=1}^{\infty}g_{\xi}((x_{i+1},y_{i+1}),(x_i,y_i)) < \infty.\]
From the definition of $g_\xi$ it follows that 
\[\sum_{i=1}^{\infty}g(x_{i+1},x_{i}) < \infty, \text{ and }\sum_{i=1}^{\infty}g(y_{i+1},y_i) < \infty.\]
From $\Sigma_g$-completeness of the space $X$ it follows $x_i \rightarrow x \in X$ and $y_i \rightarrow y \in X$, as $i \rightarrow \infty$. Since the graph of $S$ is closed then $(x,y) \in \gph S$.

\smallskip


Now, we can apply Ekeland's theorem for the function

(1) $g(x,y)$ on $ \gph S$ with the premetric function $g_\xi$;

(2) $(x_0,y_0)=(\bar{x},\bar{y})$

(3) $\varepsilon=g(\bar{x},\bar{y})$, $\lambda=\frac{g(\bar{x},\bar{y})}{r_1-1}$.

Therefore there exists $(\bar{u},\bar{v}) \in$ $\gph S$ such that \begin{equation} \label{Ek_in_grphF} 
g(x,y)+(r_1-1)g_{\xi}((x,y),(\bar{u},\bar{v})) \geq g(\bar{u},\bar{v}).
\end{equation}

Suppose $\bar{v} \notin S(\bar{u})$, then $g(\bar{u},\bar{v})>0$. Choose $0<r_1<r_2<r$ and set $t=\frac{g(\bar{u},\bar{v})}{r_2}$. From regularity of $S$ it follows $\bar{u} \in B(\bar{v},r_2t) \subset S(B(\bar{u},t))$. Therefore there exists $w \in \gph S$ such that $\bar{u} \in S(w)$ and $
g(w,\bar{u}) \leq t = \frac{g(\bar{u},\bar{v})}{r_2}$, that is 
\[r_2g(w,\bar{u}) \leq g(\bar{u},\bar{v}).\]

In \ref{Ek_in_grphF} put $(x,y)=(w,\bar{u})$.
\[g(w,\bar{u})+(r_1-1)g_{\xi}((w,\bar{u}),(\bar{u},\bar{v})) =\]
\[g(w,\bar{u})+r_1g_{\xi}((w,\bar{u}),(\bar{u},\bar{v}))-\max\{g(w,\bar{u}),\xi g(\bar{u},\bar{v})\} \leq
\]
\[r_1g_{\xi}((w,\bar{u}),(\bar{u},\bar{v})) \leq \max\{\frac{r_1}{r_2}g(\bar{u},\bar{v}),r_1 \xi g(\bar{u},\bar{v})\}<g(\bar{u},\bar{v}).\]

The last is a contradiction with \ref{Ek_in_grphF}.
\end{proof}
\begin{examp} \footnote{This example was created with the assistance of Google Gemini (June 2026 version) and verified by the author.} Let $X=\R_+$ with the usual topology, 
\[g(x,y)=\left\{\begin{array}{ll}|x-y|e^{-\frac{1}{\max\{x,y\}}}, & x \neq y,\\
0, & x=y\end{array}\right.\] and 
\[S(x)=\{x^2+x\}.\] The unique fixed point of $S$ in $X$ is 0. The regularity rate of $S$ with respect to the usual metric at 0 is 1, so we cannot apply Theorem 7.44 from \cite{Io}. But we can apply Theorem \ref{th:Fixed_point_regularity}.

Indeed, note that $g$ is a symmetric premetric and also (P4) holds. Further, we will prove that $X$ is $\Sigma_g$-complete space. Consider a sequence $\{x_i\}_{i=1}^{\infty}$ such that 
\begin{equation} \label{fin_s}
\displaystyle \Sigma_{i=1}^{\infty}g(x_{i+1},x_i)<\infty.\end{equation}
Suppose $\{x_i\}_{i=1}^{\infty}$ is divergent therefore there exists $\varepsilon>0$ such that for all $N \in \N$ there exist $m > n \geq N$
\begin{equation} \label{divergent}
|x_m-x_n| \geq \varepsilon.
\end{equation}
The function $f(t) = e^{-1/t}$ is strictly increasing for $t > 0$. Therefore,
$$\int_a^b e^{-\frac{1}{t}} dt \le (b - a) e^{-\frac{1}{b}}$$
Note that $\lim_{t \to 0^+} e^{-1/t} = 0$, so the integral above is convergent even in case $a=0$. For any two points $x_i$ and $x_{i+1}$, we establish that
$$g(x_{i+1}, x_i)=|x_{i+1}-x_i|e^{-\frac{1}{\max\{x_{i+1},x_i\}}} \ge \left| \int_{x_i}^{x_{i+1}} e^{-\frac{1}{t}} dt \right|,$$
hence
$$\sum_{i=n}^{m-1} g(x_{i+1}, x_i) \geq \sum_{i=n}^{m-1} \left| \int_{x_i}^{x_{i+1}} e^{-\frac{1}{t}} dt \right| \ge \left| \int_{x_n}^{x_m} e^{-\frac{1}{t}} dt \right| = \int_{\min\{x_n, x_m\}}^{\max\{x_n, x_m\}} e^{-\frac{1}{t}} dt.$$
Taking into account \ref{divergent} and the monotonicity of $f$ we obtain
$$\sum_{i=n}^{m-1} g(x_{i+1}, x_i) \ge \int_0^{\varepsilon} e^{-\frac{1}{t}} dt>0.$$

We now derive the regularity rate of $S$ with respect to $g$. For the map $S$ the regularity rate is equivalent to 
\[\displaystyle r = \inf_{\substack{x, w \in X \\ w \neq x}} \frac{g(S(w), S(x))}{g(w, x)}=\inf_{\substack{x, w \in X \\ w \neq x}}\frac{|x^2+x-w^2-w|e^{-\frac{1}{\max\{x^2+x,w^2+w\}}}}{|x-w|e^{-\frac{1}{\max\{x,w\}}}}.\]
Due to the symmetry of $g$ and the monotonicity of $S(x) = \{x^2 + x\}$, we assume without loss of generality that $x \ge w$. It follows $S(x) \ge S(w)$ and therefore,
\[r=\inf_{\substack{x, w \in X \\ w \neq x}} \frac{(x+w+1)e^{-\frac{1}{x^2+x}}}{e^{-\frac{1}{x}}}=e>1.\]
\end{examp}

\section*{Acknowledgements}
The author gratefully acknowledges Nadia Zlateva, Milen Ivanov and Rumen Marinov for the valuable discussions and insightful comments that contributed for the development of this work.

\section*{Funding}
This research was supported by the Bulgarian National Science Fund under Grant No. KP-06-H92/6 (December 8, 2025).

\end{document}